\documentclass{amsart}
\usepackage[utf8]{inputenc}
\usepackage{amsthm}
\usepackage{amsmath}

\usepackage{enumerate}   
\usepackage{amssymb}
\usepackage{fancyhdr} 
\usepackage[left=3.5cm,right=3.5cm,top=3.5cm,bottom=3.5cm]{geometry}
\usepackage{xcolor}
\newcommand{\norm}[1]{\left\lVert#1\right\rVert}
\usepackage{graphicx}
\usepackage{ esint }
\newtheorem{theorem}{Theorem}[section]
\newtheorem{proposition}[theorem]{Proposition}

\newtheorem{lemma}[theorem]{Lemma}
\theoremstyle{definition}

\newtheorem*{conjecture}{Conjecture}

\theoremstyle{remark}

\renewcommand\Re{\operatorname{Re}}

\newcommand{\eps}{\varepsilon}

\DeclareMathOperator{\spann}{span}
\DeclareMathOperator{\li}{li}

\usepackage{url}
\usepackage[hidelinks]{hyperref}
\hypersetup{
	colorlinks=true,
	linkcolor=cyan,
	filecolor=cyan,
	citecolor =cyan,      
	urlcolor=cyan,
}
\author{Frederik Broucke}
\thanks{F. Broucke was supported by a postdoctoral fellowship (grant number 12ZZH23N) of the Research Foundation -- Flanders.}
\address{Department of Mathematics: Analysis, Logic and Discrete Mathematics, Ghent University, {Krijgslaan} 281, 9000 Gent, Belgium}
\email{fabrouck.broucke@ugent.be}
\author{Athanasios Kouroupis}
\address{Department of Mathematical Sciences, Norwegian University of Science and Technology (NTNU), 7491 Trondheim, Norway}
\email{athanasios.kouroupis@ntnu.no}
\author{Karl-Mikael Perfekt}
\address{Department of Mathematical Sciences, Norwegian University of Science and Technology (NTNU), 7491 Trondheim, Norway}
\email{karl-mikael.perfekt@ntnu.no}
\title{A note on Bohr's theorem for Beurling integer systems}

\begin{document}
	
\begin{abstract}
Given a sequence of frequencies $\{\lambda_n\}_{n\geq1}$, a corresponding generalized Dirichlet series is of the form $f(s)=\sum_{n\geq 1}a_ne^{-\lambda_ns}$. We are interested in multiplicatively generated systems, where each number $e^{\lambda_n}$ arises as a finite product of some given numbers $\{q_n\}_{n\geq 1}$, $1 < q_n \to \infty$, referred to as Beurling primes. In the classical case, where $\lambda_n = \log n$, Bohr's theorem holds: if $f$ converges somewhere and has an analytic extension which is bounded in a half-plane $\{\Re s> \theta\}$, then it actually converges uniformly in every half-plane $\{\Re s> \theta+\varepsilon\}$, $\varepsilon>0$. We prove, under very mild conditions, that given a sequence of Beurling primes, a small perturbation yields another sequence of primes such that the corresponding Beurling integers satisfy Bohr's condition, and therefore the theorem. Applying our technique in conjunction with a probabilistic method, we find a system of Beurling primes for which both Bohr's theorem and the Riemann hypothesis are valid. This provides a counterexample to a conjecture of H. Helson concerning outer functions in Hardy spaces of generalized Dirichlet series.
\end{abstract}
\maketitle
\section{Introduction}
For an increasing sequence of positive frequencies $\lambda=\{\lambda_n\}_{n\geq 1}$, and a generalized Dirichlet series
\begin{equation*}
	f(s)=\sum\limits_{n\geq1}a_n e^{-\lambda_ns},
\end{equation*}
the abscissas $\sigma_c, \sigma_u,$ and $\sigma_a$ of point-wise, uniform, and absolute convergence are defined as in the classical theory of Dirichlet series \cite{HR64}. In this article we wish to find sets of frequencies such that the analogue of a theorem of Bohr \cite{BOH13} holds: if $\sigma_c(f) < \infty$ and $f$ has a bounded analytic extension to a half-plane $\{\Re s> \theta\}$, then $\sigma_u(f) \leq \theta$. The problem of finding frequencies for which the abscissas of bounded and uniform convergence always coincide, which originated with Bohr and Landau \cite{LAN21}, has recently been revisited \cite{BAY22,SC20} with the context of Hardy spaces of Dirichlet series in mind. Indeed, Bohr's theorem is essentially a necessity for a satisfactory Hardy space theory, see \cite[Ch. 6]{QQ20}.

An important class of frequencies were introduced by Beurling \cite{BE37}. Given an arbitrary  increasing sequence $q=\{q_n\}_{n\geq1}$, $1 < q_n\rightarrow \infty$, such that $\{\log q_n\}_{n\geq1}$ is linearly independent over $\mathbb{Q}$, we will denote by  $\mathbb{N}_{q}=\{\nu_n\}_{n\geq 1}$ the set of numbers that can be written (uniquely) as finite products with factors from $q$, ordered in an increasing manner.  The numbers $q_n$ are known as Beurling primes, and the numbers $\nu_n$ are Beurling integers. The corresponding generalized Dirichlet series are of the form
$$f(s)=\sum\limits_{n\geq 1} a_n\nu_n^{-s}.$$

There are a number of criteria to guarantee the validity of Bohr's theorem for frequencies $\{\lambda_n\}_{n\geq 1}$.  Bohr's original condition asks for the existence of $c_1, c_2>0$ such that
	\begin{equation}\label{BC}
		\lambda_{n+1}-\lambda_n \geq c_1e^{-c_2\lambda_{n+1}},\qquad n\in\mathbb{N}.
	\end{equation}
	Landau relaxed the condition somewhat: for every $\delta>0$ there should be a $c>0$ such that
	\begin{equation}\label{LC}
	 \lambda_{n+1}-\lambda_n \geq ce^{-e^{\delta\lambda_{n+1}}},\qquad n\in\mathbb{N}.
	\end{equation}
	Landau's condition was recently relaxed further by Bayart \cite{BAY22}: for every $\delta>0$ there should be a $C>0$ such that for every $n\geq1$ it holds that
	\begin{equation}\label{NC}
		\inf_{m > n} \left( \log\left(\frac{\lambda_{m}+\lambda_n}{\lambda_{m}-\lambda_n}\right)+(m-n) \right) \leq Ce^{\delta\lambda_{n}}.
	\end{equation}
For frequencies of Beurling type, $\lambda_n = \log \nu_n$, these conditions have natural reformulations. For example, Bohr's condition \eqref{BC} is equivalent to the existence of $c_1, c_2 > 0$ such that
\begin{equation}
\label{BCalt}
\nu_{n+1} - \nu_n \geq c_1 \nu_{n+1}^{-c_2}.
\end{equation}
Conditions \eqref{BC}-\eqref{BCalt} are usually very difficult to check for any given Beurling system, since they involve the distances between the corresponding Beurling integers. 
Furthermore, they often fail. This is especially true if one wants to retain properties of the ordinary integers, such as the asympotic behaviour of the counting function $N_q(x) = \sum_{\nu_n \leq x} 1$, see e.g. \cite{Gran91} for the subtleties that arise already when dealing with a finite sequence $q = (q_1, \ldots, q_N)$ of primes.

One motivation for considering Beurling integers is to investigate the properties of the $q$-zeta function
\begin{equation*}
\zeta_q(s)=\sum\limits_{n\geq 1}\nu_n^{-s}=\prod\limits_{n\geq 1}\frac{1}{1-q_n^{-s}},
\end{equation*}
and their interplay with the counting functions
\[
N_q(x)=\sum_{\nu_n\leq x}1, \qquad 	\pi_q(x)=\sum_{q_n\leq x}1. 
\]
As an example, Beurling \cite{BE37} himself showed that the condition
\begin{equation}\label{ber}
	N_q(x)=ax+ O(\frac{x}{(\log x)^\gamma}),\qquad\text{ for some }\gamma>\frac{3}{2},
\end{equation}
implies the analogue of the prime number theorem,
\begin{equation}\label{PNT}
	\pi_q(x):=\sum_{q_n\leq x}1\sim \frac{x}{\log x}.
\end{equation}
We refer to \cite{DZ16} for a comprehensive overview of further developments.

In Section 2 we begin with a preparatory result which is interesting in its own right.  It states that starting with the classical set of primes numbers we can add almost any finite sequence of Beurling primes while retaining the validity of Bohr's theorem.
\begin{theorem}\label{main}
	Let $\{p_n\}_{n\geq1}$ be the sequence of ordinary prime numbers and let $N \geq 1$. Then Bohr's condition \eqref{BC} holds for the Beurling integers generated by the primes
	$$q = \{p_n\}_{n\geq1}\bigcup\{q_j\}_{j=1}^N,$$
	for almost every choice $(q_1,\dots,q_N)\in (1,\infty)^N$.
\end{theorem}
Sequences of Beurling primes of the type considered in Theorem~\ref{main} previously appeared in \cite{Th22}.

Our next result requires more careful analysis.
\begin{theorem}\label{primesae}
	Let  $q=\{q_n\}_{n\geq 1}$ be an increasing sequence of primes such that $q_1>1$ and  $\sigma_c(\zeta_q)<\infty$. Then, for every $A > 0$ there exists a sequence of Beurling primes $\tilde{q}=\{\tilde{q}_n\}_{n\geq 1}$ for which Bohr's condition \eqref{BC} holds and 
	\begin{equation*}
|q_n-\tilde{q}_n|\leq q_n^{-A},\qquad n\in\mathbb{N}.
	\end{equation*}
\end{theorem}

Combining our techniques with a probabilistic method from \cite{BV21}, which refined previous work of Diamond, Montgomery, Vorhauer \cite{DMV06} and Zhang \cite{ZHA07}, we are able to construct a system of Beurling primes that satisfies Bohr's theorem as well as the Riemann and Lindel\"of hypothesis.
\begin{theorem}\label{th: Bohr+RH+LH}
There exists a system of Beurling primes $q=\{q_n\}_{n\geq 1}$ such that:
\begin{enumerate}[$(i)$]
	\item The Beurling zeta function $\zeta_q(s)$ has an analytic extension to $\Re s>1/2$, except for a simple pole at $s=1$. \label{analc}
	\item The Beurling zeta function has no zeros in the half-plane $\mathbb{C}_{\frac{1}{2}}=\{\Re s>1/2\}$.\label{RHZ}
	\item The prime counting function $\pi_q(x)$ satisfies 
		\begin{equation*}
			\pi_q(x)=\li (x)+O(1),
		\end{equation*}
		where $\li(x)=\int\limits_{1}^{x}(1-u^{-1})\left(\log u\right)^{-1}\,du$.\label{pnt}
	\item The integer counting function $N_{q}(x)$ satisfies
		\[
			N_{q}(x)  = ax + O_{\eps}(x^{1/2+\eps}), \quad \text{for all } \eps>0,
		\] for some $a>0$.\label{density}
	\item The corresponding Beurling integer system satisfies Bohr's condition. \label{bohrrh}
	\end{enumerate}
\end{theorem}

%Applying Theorem~\ref{primesae} to a sequence of primes satisfying the prime number theorem, we obtain a system of Beurling integers which satisfies both the prime number theorem and Bohr's condition.
%\begin{corollary}\label{RH}
%	There exists a system of Beurling primes $q=\{q_n\}_{n\geq 1}$ such that:
%	\begin{enumerate}[$(i)$]
%		\item The Riemann zeta function $\zeta_q(s)$ has an analytic extension to $\Re s>\frac{1}{2}$, except for a simple pole at $s=1$. \label{analc}
%		\item The Riemann zeta function has no zeros in the half-plane $\mathbb{C}_{\frac{1}{2}}=\{\Re s>1/2\}$.\label{RHZ}
%		\item The prime counting $\pi_q(x)$ satisfies 
%		\begin{equation*}
%			\pi_q(x)=\li (x)+O\left(\sqrt{x}\right),
%		\end{equation*}
%		where $\li(x)=\int\limits_{2}^{x}\left(\log u\right)^{-1}\,du$.\label{pnt}
%		\item Bohr's condition holds for the associated class of  generalized Dirichlet series. \label{bohrrh}
%	\end{enumerate}
%\end{corollary}
%We remark that as a direct consequence of \eqref{pnt}, we have the prime number theorem \eqref{PNT}, and also that $N_q(x)\sim ax$, where $a$ is the residue of $\zeta_q(s)$ at $s=1$, see for example \cite{DIA77,HL06}. 

The proofs of our results investigate how well ``irrational numbers" may be approximated by fractions of Beurling integers. We will comment further on this kind of Diophantine approximation problems in Section 3.

In Section 3 we will also return to the original idea behind our work. There has been an interest in studying Hardy spaces of generalized Dirichlet series since the 60s \cite{DS19, HLS97, HEL65}. However, to our knowledge, except for examples that are very closely related to the ordinary integers, there has not been any discussion of the existence of Beurling primes $q$ satisfying the prime number theorem \eqref{pnt} such that Bohr's theorem holds true for the corresponding Hardy space $\mathcal{H}^\infty_q$. This is in spite of the fact that Bohr's theorem is crucial for a meaningful theory of the Hardy spaces $\mathcal{H}^p_q$, $1 \leq p \leq \infty$. 

Since other aspects of the function theory of Hardy spaces do not depend on the choice $q$ of Beurling primes, the motivation for Theorem~\ref{th: Bohr+RH+LH} was to find a canonical Beurling system $\mathbb{N}_q$ which allows us to assume the Riemann hypothesis in the $\mathcal{H}^p_q$-theory. As a specific function theoretic application, we construct an outer function, or synonymously, a cyclic function, $f \in \mathcal{H}^2_q$ which has a zero in its half-plane of convergence,
\begin{equation}\label{eq:fdef}
f(s) = \frac{1}{\zeta_q(s + 1/2 + \varepsilon)}, \qquad  0 < \varepsilon < 1/2.
\end{equation}
The existence of such an $f$ constitutes a counterexample to a conjecture posed by Helson \cite{HEL69}.
\subsection*{Notation}
Throughout the article, we will be using the convention that $C$ denotes a positive constant which may vary from line to line. We will write that $C = C(\Omega)$ when the constant depends on the parameter $\Omega$. 

\subsection*{Acknowledgments}
The authors thank Titus Hilberdink for helpful comments and discussions.

\section{Proof of the main results}

\begin{lemma}\label{distancebc}
Suppose that $\{q_n\}_{n\geq1}$ is a Beurling system such that $d_n := \nu_{n+1} - \nu_n \gg \nu_{n+1}^{-C}$. Then, for every $\varepsilon>0$ and for almost every $q' >1$, the Beurling system $\{q_n\}_{n\geq1}\cup \{q'\}$ has a distance function satisfying
\begin{equation}\label{distfin}
d'_n = \nu'_{n+1} - \nu'_n \gg \nu_{n+1}^{-C^{'}},\qquad n\in\mathbb{N},
\end{equation}
where $C'(q',\,q)=\max\left(C,\,2\sigma_c(\zeta_q)-1+\varepsilon\right)$.
\end{lemma}
\begin{proof}
	Let $x_0 > 1$. First we will prove that the set $\mathcal{M}$ of all numbers $ q' \geq x_0$ such that there exist infinitely many triples $(j,n,m)\in\mathbb{N}^3$ with
$$\left|(q')^j-\frac{\nu_n}{\nu_m}\right|\leq \nu_n ^{-C_0}\nu_m ^{-C_0},\qquad C_0=\sigma(\zeta_q)+\varepsilon,$$
has measure zero. Since
\begin{equation*}
	\left|q'-\left(\frac{\nu_n}{\nu_m}\right)^{\frac{1}{j}}\right| \leq x_0^{1-j}\left|(q')^j-\frac{\nu_n}{\nu_m}\right|,
\end{equation*}
we have that $\mathcal{M}\subset \limsup_{m,n,j} \Omega_{m,n,j}$, where
$$\Omega_{m,n,j}=\left[\left(\frac{\nu_n}{\nu_m}\right)^{\frac{1}{j}}-x_0 ^{1-j}\nu_n^{-C_0}\nu_m^{-C_0},\, \left(\frac{\nu_n}{\nu_m}\right)^{\frac{1}{j}}+x_0 ^{1-j}\nu_n^{-C_0}\nu_m^{-C_0}\right],\qquad j, n, m \geq 1.$$
The Borel--Cantelli lemma thus shows that $|\mathcal{M}|=0$, since
\begin{equation*}
\sum\limits_{m\geq 1}\sum\limits_{n\geq 1}\sum\limits_{j\geq 1}|\Omega_{m,n,j}|\leq \frac{2x_{0}}{x_{0}-1}\zeta_q\left(C_0\right)^2<\infty.
\end{equation*}

Fix a number $q'\in[x_0,\infty)\setminus\mathcal{M}$ such that $\log q'$ is not in the (countable) set $\spann_{\mathbb{Q}}\{\log q_n\}$. Note that the set of such numbers has full measure in $[x_0,\infty)$, and that $x_0 > 1$ is arbitrary. 
By construction, there are finitely many triples $(j,n,m)$ such that
\begin{equation} \label{eq:exceptional1}
\left|(q')^j-\frac{\nu_n}{\nu_m}\right|\leq \nu_n ^{-C_0}\nu_m ^{-C_0}.
\end{equation}
For these exceptional triples, the left-hand side is at least positive, since $\log q'\notin \spann_{\mathbb{Q}}\{\log q_n\}$. Therefore
$$\left|(q')^j-\frac{\nu_n}{\nu_m}\right| \gg \nu_n ^{-C_0}\nu_m ^{-C_0}$$
for \textit{all} $(j,n,m) \in \mathbb{N}^3$.

Now we consider two arbitrary consecutive Beurling integers generated by the prime system $\{q_n\}_{n\geq1}\cup \{q'\}$, 
\begin{equation*}
\nu'_{n+1}=(q')^a\nu_m,\qquad \nu'_{n}=(q')^b\nu_l.
\end{equation*}
If $a = b$, then $l = m-1$ and 
\[\nu'_{n+1}-\nu'_n \gg \nu_m^{-C} \geq \left(\nu_{n+1}'\right)^{-C},\] by the hypothesis on the distances $d_n$ for the original Beurling system. Otherwise, if, say, $b < a$, then
\begin{equation*}
\left|\nu'_{n+1}-\nu'_n\right|= (q')^b\nu_m\left|(q')^{a-b}-\frac{\nu_l}{\nu_m}\right| \gg  \nu_l ^{-C_0}\nu_m ^{-C_0+1}(q')^b \gg \left(\nu_{n+1}'\right)^{-C^{'}}.
\end{equation*}
where $C^{'}=2\sigma_c(\zeta_q)-1+\varepsilon$.
\end{proof}
\begin{proof}[\textbf{Proof of Theorem~\ref{main}}]
The proof is a direct consequence of Lemma~\ref{distancebc}.
\end{proof}

In order to prove Bohr's theorem for more general Beurling systems, we need to control the constant in the distance estimate \eqref{distfin}, which comes from the exceptional triples satisfying \eqref{eq:exceptional1}.

\begin{proof}[\textbf{Proof of Theorem~\ref{primesae}}]
	Fix a small $\varepsilon > 0$ and $x_0\in (1+\varepsilon/2,1+\varepsilon)$. Consider first any Beurling system $\mathbb{N}_\rho = \{\nu_n\}_{n \geq 1}$ generated by Beurling primes such that $\rho_1 > 1 + \varepsilon$ and $\sigma_c(\zeta_\rho) < \infty$. For a number $\sigma_\infty > \max(2, A)$ to be chosen in a moment, let 
	$$\mathcal{N} =  \bigcup\limits_{m\geq2}\bigcup\limits_{n\geq 2}\bigcup\limits_{j\geq1} \Omega_{m, n, j},$$
    where $\Omega_{m, n, j}$ is defined as in the proof of Lemma~\ref{distancebc},
    	$$\Omega_{m,n, j}=\left[\left(\frac{\nu_n}{\nu_m}\right)^{\frac{1}{j}}-x_0 ^{1-j}\nu_n^{-\sigma_\infty}\nu_m^{-\sigma_\infty},\, \left(\frac{\nu_n}{\nu_m}\right)^{\frac{1}{j}}+x_0 ^{1-j}\nu_n^{-\sigma_\infty}\nu_m^{-\sigma_\infty}\right].$$
      Then
    $ |\mathcal{N}| \leq C(\varepsilon)\left(\zeta_{\rho}(\sigma_\infty)-1\right)^2.$ 
    Furthermore, for $x > 2$, let  
    $$I_x=[x-x^{-\frac{\sigma_\infty}{2}},x+x^{-\frac{\sigma_\infty}{2}}].$$
    Note that if $\sigma_\infty$ is sufficiently large, $\sigma_\infty \geq C(\varepsilon)$, then $I_x \cap \Omega_{m,n, j} \neq \emptyset$ only if $\nu_n \geq x/2$. Therefore
    
	$$\left|I_x\cap\mathcal{N}\right| \leq \sum\limits_{\substack{m \geq 2 \\j\geq1\\\nu_n\geq \frac{x}{2}}} |\Omega_{m,n,j}|\leq C(\varepsilon)\left(\zeta_{\rho}(\sigma_\infty)-1\right)\sum\limits_{\nu_n\geq\frac{x}{2}}\nu_n^{-\sigma_\infty}\leq C(\varepsilon)\left(\zeta_{\rho}(\sigma_\infty)-1\right) \zeta_{\rho}\left(\frac{\sigma_\infty}{4}\right)  x^{-\frac{3\sigma_\infty}{4}}.$$
		We will construct a sequence of Beurling systems such that
	\begin{equation} \label{eq:bddquant}
	\left(\zeta_{\rho}(\sigma_\infty)-1\right) \zeta_{\rho}\left(\frac{\sigma_\infty}{4}\right) \leq 1
    \end{equation}
   for the number $\sigma_\infty > 0$, still to be chosen later. Therefore
    \begin{equation} \label{eq:existenceineq}
    \left|I_x\cap\mathcal{N}\right| \leq C(\varepsilon) x^{-\frac{\sigma_\infty}{4}} |I_x|,
    \end{equation}
    We conclude that whenever $x$ is sufficiently large, $I_x \not \subset \mathcal{N}$.
	
	To include triples where $\nu_n$ or $\nu_m$ equals one in our considerations, we increase the power $\sigma_\infty$. The inequality
	\begin{equation}\label{eq:powerincr}
	\left|x^j-\frac{\nu_n}{\nu_m}\right|\leq \nu_n^{-3\sigma_\infty} \nu_m^{-3\sigma_\infty}
	\end{equation}
	implies, whenever $x \geq x_0$, that
	$$\left|x - \left(\frac{\nu_n}{\nu_m}\right)^{\frac{1}{j}}\right| \leq x_0^{1-j}\left|x^j-\frac{\nu_n^2\nu_m }{\nu_m^2\nu_n}\right|\leq x_0^{1-j} \left(\nu_m^2\nu_n\right)^{-\sigma_\infty}\left(\nu_n^2\nu_m\right)^{-\sigma_\infty}.$$
	Therefore $\mathcal{M} \subset \mathcal{N}$, where $\mathcal{M}$ this time denotes the set of all $x \geq x_0$ for which there exists an exceptional triple $ (j,n,m) \in \mathbb{N}^3$ such that \eqref{eq:powerincr} holds.
	
	Now let $q$ be a sequence of primes in the statement of Theorem~\ref{primesae}, assuming that $\varepsilon < q_1 - 1$. As described, we will only be able to effectively apply \eqref{eq:existenceineq} when $x$ is sufficiently large, say, $x \geq B = B(\varepsilon) = C(\varepsilon)^4 + 2$, where $C(\varepsilon)$ in this instance refers to the same constant as in \eqref{eq:existenceineq}.  Let $N$ be such that $\{q_1, \ldots, q_N\} = (1, B) \cap q$.   Then, as a corollary of Theorem~\ref{main}, we already know that there exists an increasing finite sequence of primes $\{\tilde{q}_1, \ldots \tilde{q}_N\}$, $\tilde{q}_1 > 1$, such that $|q_j - \tilde{q}_j| \leq q_j^{-A}$, $j=1, \ldots, N$, and such that Bohr's condition holds for $\{\nu_n^{(N)}\}_{n\geq1}=\mathbb{N}_{\{\tilde{q}_1, \ldots \tilde{q}_N\}}$.
	 Further, we choose $\sigma_\infty$ so large that
	 		\begin{equation*}
	 	\left|\nu_{n+1}^{(N)}-\nu_n^{(N)}\right|\geq \left(\nu_{n+1}^{(N)}\right)^{-6\sigma_\infty},\qquad n\in\mathbb{N},
	 \end{equation*}
 and
	\begin{equation}\label{eq:sigmainfchoice}
	\left(\zeta_{q'}(\sigma_\infty)-1\right) \zeta_{q'}\left(\frac{\sigma_\infty}{4}\right) \leq 1, \quad q' = \{\tilde{q}_1, \ldots \tilde{q}_N, q_{N+1} - 1, q_{N+2} - 1, q_{N+3} - 1, \ldots\}.
	\end{equation}
	This is made possible by the hypothesis that $\sigma_c(\zeta_q) < \infty$, since
		\begin{equation*}
		\zeta_{q'}(\sigma)\leq \prod\limits_{j\geq 1}\frac{1}{1-(q_j')^{-\sigma}}\leq \zeta_q\left(\frac{\sigma}{C}\right),\qquad \sigma>0,\qquad C\geq \sup_{n \geq 1}\frac{\log (q_n)}{\log(q'_n)}.
	\end{equation*}
	
	From here we proceed by induction. Suppose that $\tilde{q}_1, \ldots \tilde{q}_k$ have been chosen, where $k \geq N$, with corresponding Beurling integers $\{\nu_n^{(k)}\}_{n\geq1}=\mathbb{N}_{\{\tilde{q}_{n}\}_{n=1}^k}$ satisfying  that	$$\left|\nu_{n+1}^{(k)}-\nu_n^{(k)}\right|\geq \left(\nu_{n+1}^{(k)}\right)^{-6\sigma_\infty}.$$
	 We apply the preceding discussion to the Beurling primes $\rho = \{\tilde{q}_1, \ldots \tilde{q}_k\}$ and $x = q_{k+1}$, concluding that there exists a number $\tilde{q}_{k+1} \in I_{q_{k+1}}$ such that
	$$\left|\tilde{q}_{k+1}^j-\frac{\nu^{(k)}_n}{\nu^{(k)}_m}\right| \geq \left(\nu_n^{(k)}\right)^{-3\sigma_\infty} \left(\nu_m^{(k)}\right)^{-3\sigma_\infty}, \qquad (j, n, m) \in \mathbb{N}^3.$$
	By the same argument as in the last paragraph of the proof of Theorem~\ref{main} the Beurling system $\{\nu_n^{(k+1)}\}_{n\geq1}=\mathbb{N}_{\{\tilde{q}_n\}_{n=1}^{k+1}}$, then satisfies that
	\begin{equation*}
		\left|\nu_{n+1}^{(k+1)}-\nu_n^{(k+1)}\right|\geq \left(\nu_{n+1}^{(k+1)}\right)^{-6\sigma_\infty},\qquad n\in\mathbb{N}.
	\end{equation*}
	
	At each step of the construction, \eqref{eq:sigmainfchoice} ensures that \eqref{eq:bddquant} holds.
	We hence obtain a sequence $\tilde{q}=\{\tilde{q}_n\}_{n\geq 1}$, satisfying that $|\tilde{q}_n-q_n|\leq q_n^{-\frac{\sigma_\infty}{2}}$ as well as Bohr's condition \eqref{BC}, specifically,
	\begin{equation*}
	\left|\tilde{\nu}_{n+1}-\tilde{\nu}_{n}\right|\geq\left(\tilde{\nu}_{n+1}\right)^{-6\sigma_\infty},\qquad n\in\mathbb{N}.
	\end{equation*}
where $\{\tilde{\nu}_n\}_{n\geq1}=\mathbb{N}_{\tilde{q}}$.
\end{proof}

To prove Theorem \ref{th: Bohr+RH+LH}, we shall combine the proof of Theorem~\ref{primesae} with the probabilistic construction of \cite[Theorem 1.2]{BV21}.  Let
\[
F(x) = \li(x) = \int\limits_{1}^{x}\frac{1-u^{-1}}{\log u}\,du
\]
and set $x_{n} = F^{-1}(n)$. We select the $n$th Beurling prime $q_{n}$ randomly from the interval $[x_{n},x_{n+1}]$ according to the probability measure $d\li(x)\rvert_{[x_{n},x_{n+1}]}$. That is, we consider a sequence of independent random variables $Q_{n}$, representing the coordinate functions $(q_{1}, q_{2}, \dotsc) \mapsto q_{n}$, with cumulative distribution function $\int\limits_{x_{n}}^{x}d\li(u) = \li(x) - n$, $x_{n} \le x \le x_{n+1}$. Formally, the probability space is $X=\prod_{n=1}^{\infty}[x_{n},x_{n+1}]$, and by appealing to Kolmogorov's extension theorem, we can equip $X$ with a probability measure $dP$ such that 
\[
P\biggl(A \times \prod_{n=k+1}^{\infty}[x_{n},x_{n+1}]\biggr) = \int_{A}d\li(u_{1})\dotsm d\li(u_{k}), \quad \text{if } A \subseteq [x_{1},x_{2}]\times \dotsb \times [x_{k},x_{k+1}].
\] 

\begin{proof}[\textbf{Proof of Theorem~\ref{th: Bohr+RH+LH}}]
	Let $A>1$. We will show the existence of a sequence of Beurling primes $q=\{q_{n}\}_{n\ge1}$ generating integers $\mathbb{N}_{q} = \{\nu_{n}\}_{n\ge1}$ with the following properties:
	\begin{enumerate}[(a)]
		%\item The first prime satisfies $q_{1}\ge1+\eps$.
		\item The Beurling zeta function $\zeta_{q}(s)$ can be written as
		\[
		\zeta_{q}(s) = \frac{se^{Z(s)}}{s-1},
		\]
		where $Z(s)$ is an analytic function in $\{\Re s > 1/2\}$ which in every closed half-plane $\{\Re s \ge \sigma_0\}$, $\sigma_{0}>1/2$, satisfies
		\[
		|Z(s)| \ll_{\sigma_{0}} \sqrt{\log(|t|+2)}, \qquad s = \sigma + it.
		\]\label{Z(s)}
		\item The Beurling integers satisfy $|\nu_{n}-\nu_{m}| \ge (\nu_{n}\nu_{m})^{-A}$ whenever $n\neq m$.
	\end{enumerate}
	
	Suppose that $\nu, \mu$ are relatively prime Beurling integers, $(\nu, \mu) = 1$, satisfying $|\nu-\mu| < (\nu\mu)^{-A}$. Let $q_{k}$ be the largest prime factor of $\nu\mu$. Without loss of generality $q_{k} \mid \nu$ and $q_{k}\nmid \mu$, and let $j$ be such that $\nu = q_{k}^{j}\nu'$ where $q_{k} \nmid \nu'$. Then
	\[
	\left\lvert q_{k}^{j} - \frac{\mu}{\nu'}\right\rvert < \frac{1}{q_{k}^{Aj}(\nu')^{1+A}\mu^{A}}, 
	\]
	so that
	\[
	\left\lvert q_{k} - \Bigl(\frac{\mu}{\nu'}\Bigr)^{1/j}\right\rvert < \frac{1}{q_{k}^{(A+1)j-1}(\nu')^{1+A}\mu^{A}} \le \frac{1}{x_{k}^{(A+1)j-1}(\nu')^{1+A}\mu^{A}}.
	\]
	
	This motivates the following definitions. For $k, j \ge1$ and 
	\[
	q_{1}\in [x_{1}, x_{2}], q_{2}\in [x_{2},x_{3}], \dotsc, q_{k-1}\in [x_{k-1},x_{k}],\]
	 we set
	\begin{multline*}
	\mathcal{M}_{k,j}(q_{1}, \dotsc, q_{k-1}) \\
	= \bigcup_{\nu, \mu\in \mathbb{N}_{(q_{1},\dotsc, q_{k-1})}}\left[\Bigl(\frac{\mu}{\nu}\Bigr)^{1/j}-\frac{1}{x_{k}^{(A+1)j-1}\nu^{1+A}\mu^{A}},\Bigl(\frac{\mu}{\nu}\Bigr)^{1/j} + \frac{1}{x_{k}^{(A+1)j-1}\nu^{1+A}\mu^{A}}\right],
	\end{multline*}
	and we consider the events
	\[
	B_{k,j} = \bigl\{(q_{1},q_{2},\dotsc) : q_{k}\in \mathcal{M}_{k,j}(q_{1}, \dotsc, q_{k-1})\bigr\}.
	\]
	Denoting $\mathcal{X} = (x_{1}, x_{2},\dotsc)$, the Lebesgue measure of $\mathcal{M}_{k,j}(q_{1}, \dotsc, q_{k-1})$ is bounded by $2\zeta_{\mathcal{X}}(1+A)\zeta_{\mathcal{X}}(A)x_{k}^{1-(A+1)j}$. Note that $\pi_{\mathcal{X}}(x) = \li(x)+O(1)$, and so  $\zeta_{\mathcal{X}}(s)$ has abscissa of convergence $1$. Hence, for the probability of $B_{k,j}$ we have
	\begin{align*}
	P(B_{k,j}) 	&= \int\limits_{x_{1}}^{x_{2}}d\li(u_{1})\int\limits_{x_{2}}^{x_{3}}d\li(u_{2})\dotso\int\limits_{x_{k-1}}^{x_{k}}d\li(u_{k-1})\int\limits_{[x_{k},x_{k+1}] \cap \mathcal{M}_{k,j}(u_{1}, \dotsc, u_{k-1})} d\li(u_{k}) \\
	&\le \frac{2\zeta_{\mathcal{X}}(1+A)\zeta_{\mathcal{X}}(A)}{x_{k}^{(A+1)j-1}}.
	\end{align*}
	In particular, $\sum_{k=1}^{\infty}\sum_{j=1}^{\infty}P(B_{k,j}) < \infty$. 
	
	We also consider the events
	\[
	A_{k,m} = \biggl\{(q_{1}, q_{2}, \dotsc): \Bigl\lvert\sum_{n=1}^{k}q_{n}^{-i m}-\int\limits_{x_{1}}^{x_{k}}u^{-im}d\li(u)\Bigr\rvert \ge 8\sqrt{\frac{x_{k}}{\log x_{k}}}\bigl(\sqrt{\log x_{k}}+\sqrt{\log m}\bigr)\biggr\}.
	\]
	In the proof of \cite[Theorem 1.2]{BV21} it was shown that also $\sum_{k=1}^{\infty}\sum_{m=1}^{\infty}P(A_{k,m}) < \infty$. Hence, by the Borel--Cantelli lemma we have with probability $1$ that the sequence $(q_{1}, q_{2}, \dotsc)$ is contained in only finitely many of the sets $A_{k,m}$ and $B_{k,j}$. Take any such sequence $q$.  
	
	Note that by construction, $\pi_{q}(x) = \li(x)+O(1)$. As $q$ is contained in only finitely many $A_{k, m}$, we obtain that %as in \cite{BV21} that
	\begin{equation}
	\label{exp sum}
	\left\lvert\sum_{q_{n}\le x}q_{n}^{-it}-\int\limits_{x_{1}}^{x}u^{-it}\,d\li(u)\right\rvert \ll \sqrt{\frac{x}{\log(x+1)}}\bigl(\sqrt{\log(x+1)}+\sqrt{\log(|t|+1)}\bigr), \quad x\ge1, \quad t\in\mathbb{R}.
	\end{equation}
	For $x=x_{k}$ and $t=m\in \mathbb{Z}$ this is clear. If $x\in (x_{k}, x_{k+1})$, then both terms in the absolute value of \eqref{exp sum} change by at most $O(1)$ upon replacing $x$ by $x_{k}$. Hence the bound also holds for any $t = m\in \mathbb{Z}$ and arbitrary $x\ge1$. To obtain the bound for $t\in(m,m+1)$, we write
	\[
	\sum_{q_{n}\le x}q_{n}^{-it} = \int\limits_{x_{1}}^{x}u^{-i(t-m)}\,d\biggl(\sum_{q_{n}\le u}q_{n}^{-im}\biggr), \quad 
	\int\limits_{x_{1}}^{x}u^{-it}\,d\li(u) = \int\limits_{x_{1}}^{x}u^{-i(t-m)}\,d\biggl(\int\limits_{x_{1}}^{u}v^{-im}\,d\li(v)\biggr),
	\]
and integrate by parts.
	
	The bound \eqref{exp sum} implies (\ref{Z(s)}). Indeed, setting $\Pi_{q}(x) = \pi_{q}(x) + \frac{\pi_{q}(x^{1/2})}{2}+\dotsb$, we have 
	\[
	\log\zeta_{q}(s) = \int\limits_{1}^{\infty}x^{-s}\,d\Pi_{q}(x) = \log\frac{s}{s-1} + \int\limits_{1}^{\infty} x^{-s}\,d\bigl(\pi_{q}(x)-\li(x)\bigr) + \int\limits_{1}^{\infty}x^{-s}\,d\bigl(\Pi_{q}(x)-\pi_{q}(x)\bigr).
	\]
	Let $\sigma\ge\sigma_{0}>1/2$. In the second integral we integrate by parts and use \eqref{exp sum} to see that it is $O_{\sigma_{0}}\bigl(\sqrt{\log(|t|+2)}\bigr)$. The third integral is $O_{\sigma_{0}}(1)$ for $\sigma\ge\sigma_{0}>1/2$, since $\Pi_{q}(x)-\pi_{q}(x)$ is non-decreasing and $\ll \sqrt{x}/\log x$.
	
	Now let $k_{1}, k_{2}, \dotsc, k_{l}$ be the exceptional integers of the construction. Then $q\notin B_{k,j}$ for all $j$ and $k\neq k_{1}, \dotsc, k_{l}$. We simply remove the corresponding Beurling primes from the system: $\tilde{q} = q \setminus\{q_{k_{1}}, \dotsc, q_{k_{l}}\}$. As $\zeta_{\tilde{q}}(s) = \zeta(s)(1-q_{k_{1}}^{-s})\dotsm(1-q_{k_{l}}^{-s})$, (\ref{Z(s)}) remains valid for $\zeta_{\tilde{q}}(s)$. Finally, every two distinct Beurling integers $\nu_{n}\neq\nu_{m}$ from $\mathbb{N}_{\tilde{q}}$ satisfy $|\nu_{n}-\nu_{m}| \ge (\nu_{n}\nu_{m})^{-A}$. For if this were not the case, then, by the argument at the beginning of the proof, the largest prime factor $q_{k}$ of $\nu_{n}\nu_{m}/(\nu_{n}, \nu_{m})^{2}$ would be contained in some $\mathcal{M}_{k,j}(q_{1}, \dotsc, q_{k-1})$, which is impossible by construction.
	
	Every point in Theorem \ref{th: Bohr+RH+LH} has now been proven, except for \textit{(\ref{density})}. However, since $\zeta_{\tilde{q}}(s)$ is of zero order in $\{\Re s > \sigma_0\}$ for every $\sigma_0 > 1/2$ by (\ref{Z(s)}), this follows from a standard application of Perron inversion, see \cite{HL06,ZHA07}.
\end{proof}

\section{Further discussion}
\subsection*{Diophantine approximation and Beurling integers}
Using the Borel--Cantelli theorem to study the irrationality of real numbers is a standard technique of Diophantine approximation. The irrationality measure $\mu(x)$ of a real number $x\in\mathbb{R}$ is defined as the infimum of the set
$$R_x=\left\{r>0:\left|x-\frac{m}{n}\right|<\frac{1}{n^r} \text{ for at most finitely many pairs }(m,n)\in\mathbb{N}\times\mathbb{N}\right\}.$$
For a Beurling system $\mathbb{N}_{q}=\{\nu_n\}_{n\geq 1}$, we may also introduce the irrationality measure $\mu_q(x)$ of a real number $x\in\mathbb{R}$ as the infimum of the set
$$R_x=\left\{r>0:\left|x-\frac{\nu_m}{\nu_n}\right|<\frac{1}{\nu_n ^r} \text{ for at most finitely many pairs }(m,n)\in\mathbb{N}\times\mathbb{N}\right\}.$$
Then, by slightly modifying the proof of Lemma~\ref{distancebc}, we obtain the following proposition.
\begin{proposition}\label{paok}
	Let $q=\{q_n\}_{n\geq1}$ be a sequence of Beurling primes with $\sigma_c(\zeta_q)< \infty$. Then, for almost every $x\in\mathbb{R}$, it holds that
	\begin{equation*}\label{eq:paok}
	\mu_q(x)\leq 2\sigma_c(\zeta_q).
	\end{equation*}
\end{proposition}
In the classical case, Dirichlet's approximation theorem therefore implies that  $\mu(x)=2$ for almost every $x\in \mathbb{R}$. We also recall Roth's theorem \cite{DR55}, which states that $\mu(x)=2$ for every algebraic irrational number. It would be very interesting to develop corresponding results in the context of Beurling integers.

\subsection*{Hardy spaces of Dirichlet series and a conjecture of Helson}
For a sequence $q$ of Beurling primes, we introduce the Hardy space $\mathcal{H}_q^2$ as
$$\mathcal{H}_q^2=\left\{f(s)=\sum_{n\geq1} a_n \nu_n^{-s} \, :\, \norm{f}_{\mathcal{H}_q^2}^2=\sum_{n\geq1}|a_n|^2< \infty \right\}.$$
More generally, for $1 \leq p < \infty$, we define $\mathcal{H}_q^p$ as the completion of polynomials (finite sums $\sum a_n \nu_n^{-s}$) under the Besicovitch norm 
\begin{equation*}
	\norm{P}_{\mathcal{H}_q^p}:= \left(\lim\limits_{T\rightarrow\infty} \frac{1}{2T} \int\limits_{-T}^{T} |P(it)|^p\,dt\right)^{\frac{1}{p}}.
\end{equation*}
The function theory of these spaces originated with Helson \cite{HEL65}, and was, in the distuingished case where $q$ is the sequence of ordinary primes, continued in very influential papers of Bayart \cite{BAY02} and Hedenmalm, Lindqvist, and Seip \cite{HLS97}. More generally, there is a developing theory of Hardy spaces of Dirichlet series $\sum a_n e^{-\lambda_n s}$ whose frequencies are related to other groups than $\mathbb{T}^\infty$, but we shall restrict our attention to frequencies given by Beurling primes. A cornerstone of the theory is that there is a natural multiplicative linear isometric isomorphism between $\mathcal{H}^p_q$ and the Hardy space $H^p(\mathbb{T}^\infty)$ of the infinite torus \cite{DS19, HEL69}.  However, more is needed in order to identify $H^\infty(\mathbb{T}^\infty)$ with $\mathcal{H}^\infty_q$, the space of Dirichlet series $\sum a_n \nu_n^{-s}$ which converge to a bounded function in $\mathbb{C}_0 = \{\Re s > 0\}$. In fact, Bohr's condition is typically used in order to establish this isomorphism \cite{SC20}.   

In identifying $\mathcal{H}^p_q$ with $H^p(\mathbb{T}^\infty)$ one is naturally led to consider twisted Dirichlet series
$$f_\chi(s) = \sum_{n\geq1} a_n \chi(\nu_n) \nu_n^{-s},$$
where a point $\chi \in \mathbb{T}^\infty$ is interpreted as the completely multiplicative character $\chi \colon \mathbb{N}_q \to \mathbb{T}$ such that $\chi(q_n) = \chi_n$. Helson \cite{HEL69} proved that if $f \in \mathcal{H}^2_q$ and the associated frequencies satisfy Bohr's condition, then $f_\chi(s)$ converges in $\mathbb{C}_0$ for almost every $\chi \in \mathbb{T}^\infty$. Helson went on to make a conjecture, which we state only in the special case that the frequencies correspond to a Beurling system.
Recall that $f \in \mathcal{H}^2_q$ is said to be outer (or cyclic) if $\left\{fg:\, g\in \mathcal{H}_q^\infty\right\}$ is dense in $\mathcal{H}_q^2$.
\begin{conjecture}
	If $\mathbb{N}_q$ is a Beurling system that satisfies Bohr's condition and $f$ is outer in $\mathcal{H}_q^2$, 
	then $f_\chi$ never has any zeros in its half-plane of convergence.
\end{conjecture}

Suppose now that the Beurling primes $q$ are chosen as in Theorem~\ref{th: Bohr+RH+LH}, so that we have the Riemann hypothesis at our disposal, and consider the Dirichlet series 
$$f(s) = \frac{1}{\zeta_q(s + 1/2 + \varepsilon)},$$ 
for some $0 < \varepsilon < 1/2$. Through a routine calculation with coefficients, one checks that $f, f^2, 1/f, 1/f^2 \in \mathcal{H}^2_q$. Therefore, there are polynomials $p_n$ which converge to $1/f$ in $\mathcal{H}^4_q$, so that
$$\|1-p_n f\|_{\mathcal{H}^2_q} \leq \|f\|_{\mathcal{H}^4_q} \|1/f - p_n\|_{\mathcal{H}^4_q}  \to 0, \qquad n \to \infty.$$
Thus, $f$ is outer. On the other hand, it has a zero at $s = 1/2 - \varepsilon$. To disprove Helson's conjecture it only remains to prove that $f$ converges in $\mathbb{C}_0$.
\begin{proposition}
The reciprocal $1/\zeta_q$ of the Beurling zeta function converges in $\{\Re s>1/2\}$.
\end{proposition}
\begin{proof}
The reciprocal $1/\zeta_q$ is of zero order in $\{\Re s> 1/2\}$. This is clear from the proof of Theorem~\ref{th: Bohr+RH+LH}, but it is also well known that this can be deduced from \textit{(\ref{pnt})} and \textit{(\ref{density})} 
using the Borel--Carathéodory and the Hadamard three circles theorems, 
\begin{equation}
\log\zeta_q(\sigma+it)=O\left((\log|t|)^\alpha\right),\qquad\alpha\in(0,1),
\end{equation}
uniformly for $\frac{1}{2}<\sigma_0\leq \sigma\leq 1$. See for example \cite[Theorem~2.3]{HL06} or \cite[Theorem~14.2]{TIT86}. Since we have Bohr's condition, the standard argument \cite[Section 9.44]{TIT58} with Perron's formula then shows that $1/\zeta_q$ is convergent in the half-plane where it is analytic with zero order.\qedhere
\end{proof}

\bibliographystyle{amsplain-nodash} 
\bibliography{rfrn} 
\end{document}